\documentclass[11pt]{amsart}

\usepackage{amscd,amsmath,latexsym,amsthm,amsfonts,amssymb,graphicx,color,geometry,hyperref}

\usepackage[utf8]{inputenc}
\usepackage{amsmath}

\usepackage{enumitem}
\newtheorem{claim}{Claim}
\usepackage{placeins}
\usepackage[all]{xy}
\usepackage{biblatex}
\usepackage{nicefrac}
\usepackage{comment}
\usepackage[normalem]{ulem}
\usepackage{xcolor}
\usepackage{enumitem}
\usepackage{thmtools}
\usepackage{hyperref}
\usepackage{tikz}
\usetikzlibrary{arrows.meta, positioning}
\usepackage[noabbrev]{cleveref}			
\usepackage{pifont}
\usepackage{comment}
\usepackage{etoolbox}
\usepackage{datetime}
\usepackage[scr]{rsfso}
\usepackage{enumitem}

\newtheorem{definition}{Definition}[section]

\theoremstyle{plain}

\newtheorem{theorem}[definition]{Theorem}
\newtheorem{proposition}[definition]{Proposition}
\newtheorem{lemma}[definition]{Lemma}
\newtheorem{teo}{Theorem}

\newtheorem{corollary}[definition]{Corollary}

\theoremstyle{remark}
\newtheorem{remark}[definition]{Remark}

\theoremstyle{plain}

\newtheorem{ques}{Question}

\newcommand\restr[2]{{
		\left.\kern-\nulldelimiterspace 
		#1 
		\vphantom{\big|} 
		\right|_{#2} 
}}

\begin{document}
\title[]{Shadowing phenomenon for composition operators on the Hardy space $H^2(\mathbb{D})$}

\author{Artur Blois, Ben-Hur Eidt, Paulo Lupatini and Osmar R. Severiano} 

\address[A. Blois]{IMECC, Universidade Estadual de Campinas, Campinas, Brazil}
\email{blois@ime.unicamp.br}
\thanks{Artur Blois is supported by the Coordenação de Aperfeiçoamento de Pessoal de Nível Superior – Brasil (CAPES) – Finance Code 001.}

\address[B. H. Eidt]{ICMC, Universidade de São Paulo, São Carlos, Brazil} 
\thanks{Ben-Hur Eidt is supported by the São Paulo Research Foundation (FAPESP), Brasil. Process Number 2024/14727-4}
\email{benhur96dt@gmail.com}    

\address[P. Lupatini]{IMECC, Universidade Estadual de Campinas, Campinas, Brazil}
\email{lupatini@ime.unicamp.br}

\address[O. R. Severiano]{ IMECC, Universidade Estadual de Campinas, Campinas, Brazil}
\email{osmar.rrseveriano@gmail.com}
\thanks{Osmar R. Severiano would like to thank the Universidade Estadual de Campinas (Unicamp) for the institutional support and the opportunity to develop this research within the Postdoctoral Researcher Program (PPPD)}	
	
	
\begin{abstract}
Let $\phi$ be a holomorphic self-map of the open unit disk $\mathbb{D}.$ In this article, we study the shadowing phenomenon for composition operators $C_{\phi}f=f\circ \phi$ on the Hardy space $H^2(\mathbb{D}).$ We characterize all the composition operators induced by linear fractional self-maps of $\mathbb{D}$ that have the positive shadowing property. 
\end{abstract}
	
\maketitle

\section{Introduction}

The study of linear dynamics, especially in the last decades, has developed considerably and many connections with other areas of mathematics have been established, such as ergodic theory, number theory and geometry of Banach spaces \cite{bayart2009dynamics, grosse2011linear}. For instance, chaos, chain recurrence, structural stability, and the shadowing property have been studied in a series of papers \cite{ Bernardes_Messaoudi_shadowingprincipal, Bernardes_Alfred_Shadowing2024,  Bernardes-Cirilo, Favaro,, Nilson_Bonilla_Muller_Peris_LiYorkechaos}.

A fundamental notion in dynamical systems is the \textit{shadowing property}, which comes from the work of Sina\u{\i} \cite{Sinai} and Bowen \cite{Bowen}. The shadowing property means that near a pseudo-orbit there exists a real orbit. In the context of linear dynamics, it is well known that every invertible \textit{hyperbolic} bounded linear operator has the shadowing property \cite{Morimoto, Ombach}. Another important idea in dynamical systems is the concept of \textit{expansivity}, which was introduced by Utz \cite{Utz} and studied on Banach spaces in \cite{Bernardes_Messaoudi_shadowingprincipal, Bernardes-Cirilo}. Several authors have become interested in these topics, either by looking for examples of maps that are hyperbolic, expansive or possess the shadowing property or by establishing general criteria that allow one to determine when a map has such properties. For instance, complete characterizations of expansive and uniformly expansive weighted shifts on classical Banach sequence spaces were obtained in \cite{Bernardes-Cirilo}, while the characterization of those that satisfy the shadowing property was obtained in \cite{Bernardes_Messaoudi_shadowingprincipal}. Recently, Bernardes and Messaoudi  established the following relation between these three concepts. \textit{If $T$ is an invertible bounded linear operator on a Banach space, then $T$ is hyperbolic if and only if $T$ is expansive and has the shadowing property} \cite[Theorem 1]{Bernardes_Messaoudi_shadowingprincipal}. In the same work, and allowing non-invertible operators, they established the following. \textit{Let $T$ be a bounded linear operator on a Banach space, then $T$ is hyperbolic with spectrum contained in $\mathbb{C}\backslash \overline{\mathbb{D}} $ if and only if $T$ is positively expansive and has the positive shadowing property} \cite[Proposition 5]{Bernardes_Messaoudi_shadowingprincipal}.

Inspired by these results, it is natural to ask what the situation is for more concrete operators; this is where composition operators arise. Let $X$ be a space of complex-valued functions on the set $\Omega$ and $\phi:\Omega\to \Omega$ be a function. The \textit{composition operator} $C_{\phi}$ with \textit{symbol} $\phi$ is defined by $C_{\phi}f=f\circ \phi$ for any $f\in X.$ Operators of this type have been studied on a variety of spaces \cite{Cowen, Duren}, and as is well known, there are several non-trivial questions in the study of composition operators, such as boundedness, cyclicity, compactness, hypercyclicity, spectra and spectral radii and more recently, properties such as expansivity and shadowing property. For example, if $X$ is the $L^p$ space, then recent contributions on linear dynamics  can be found in \cite{Bernardes-Pires, Martina_Darji_Daniello_shadowingandghinL^p,Darji_Pires_Benito_compo}.

Phenomena concerning the orbit behavior of bounded linear fractional composition operators (that is, whose inducing symbols are linear fractional transformations) have been studied on several holomorphic function spaces, with special interest in the Hardy spaces of the open unit disk $H^2(\mathbb{D})$ and of the open right half-plane $H^2(\mathbb{C}_+),$ and in the weighted Dirichlet space $\mathcal{S}_{\nu}$. In this context, Bourdon and Shapiro \cite{Shapiro_Bourdon_Livro_Ciclicidade} completely characterized which linear fractional composition operators are cyclic and which are hypercyclic on $H^2(\mathbb{D})$; Gallardo-Gutiérrez and Montes-Rodríguez \cite{Eva-Afonso} did the same on each $\mathcal{S}_{\nu}$; and Noor and Severiano \cite{Noor-Severiano} obtained the corresponding characterization on $H^2(\mathbb{C}_+)$.

Recently, the (positive) shadowing property for linear fractional composition operators on $H^2(\mathbb{C}_+)$ was studied in \cite{Carlos_Javier_ShadowingHalfPlane}. One of the advantages of the half-plane setting is that some composition operators are hyperbolic (in the linear dynamics sense), expansive or normal. These three concepts provide nice ways to approach the shadowing property using general results. The situation on $H^2(\mathbb{D})$ is more complicated, since every composition operator fixes the reproducing kernel function $k_0$. Consequently, such operators are neither hyperbolic nor expansive, which suggests that a different approach is necessary.

 The goal of this paper is to add the positive shadowing property to the list of properties which are completely understood for linear fractional composition operators on $H^2(\mathbb{D}).$ This work can be seen as a contribution to the composition operator program (see the introduction of \cite{Shapiro_Bourdon_Livro_Ciclicidade}).

The paper is organized as follows. In Section \ref{section 2}, we  fix the notation and  recall a few results which will be important in our work. We finish this section by presenting our main result, Theorem \ref{Theorem}, which gives the complete characterization of all linear fractional composition operators on $H^2(\mathbb{D})$ that have the positive shadowing property. In Section \ref{section 3}, we provide the proof of our main result. We prove it by considering four cases separately; the ideas underlying these cases are independent of one another. In particular, we show that $C_{\phi}$ is generalized hyperbolic when $\phi$ is a hyperbolic non-automorphism of type I. This furnishes an example of a composition operator on $H^2(\mathbb D)$ with the shadowing property which is not hyperbolic.  Finally, in Section \ref{section 4}, we briefly comment on which ideas can be adapted to $H^p(\mathbb{D})$ for $1 \leq p \leq \infty$. Our main results are summarized in the following table. 
\FloatBarrier
\begin{table}[htb!]
\caption{Positive shadowing property for linear fractional composition operators} 
\centering
\begin{tabular}{|c|c|c|c|}
\hline
\textbf{Type of $\phi$}  & \textbf{Fixed points} & \textbf{ Canonical form of $\phi$} & \textbf{Pos. shad. }\\
      \hline
      EA   & $0,\infty$
      & $\displaystyle \phi(z)=\omega z,\quad \omega\in\mathbb{T} \setminus \{1\}$
      & \ding{55}  (Cor. \ref{elliptic-loxodormic-hyperbolic non-automorphism of type II}) \\
      \hline 

HA& $1,-1$
      & $\displaystyle \phi(z)=\frac{r+z}{1+rz},\quad 0<r<1$
      & \ding{52} (Prop. \ref{hyperbolic automorphism}) \\[2mm]
\hline

HNA I & $1,\infty$
      & $\displaystyle \phi(z)=rz+1-r,\quad 0<r<1$
      & \ding{52} (Cor. \ref{hyperbolic non-automorphism of type I}) \\[2mm]
\hline
HNA II & $0,1$
      & $\displaystyle \phi(z)=\frac{rz}{1-(1-r)z},\quad 0<r<1$ & \ding{55}  (Cor. \ref{elliptic-loxodormic-hyperbolic non-automorphism of type II}) \\[2mm]
\hline
LOX  & $c\in\mathbb{D},\infty$
      & $\displaystyle \phi(z)=a(z-c)+c,\quad |a|+|1-a||c|\le 1$
      & \ding{55} (Cor. \ref{elliptic-loxodormic-hyperbolic non-automorphism of type II})\\
\hline
PA & $1$
      & $\displaystyle \phi(z)=\frac{(2-a)z+a}{-az+2+a},\quad \mathrm{Re}(a)=0, \,\ a \neq 0$
      & \ding{55} (Th. \ref{parabolic}) \\[2mm]
\hline
PNA & $1$
      & $\displaystyle \phi(z)=\frac{(2-a)z+a}{-az+2+a},\quad \mathrm{Re}(a) > 0$
      & \ding{55} (Th. \ref{parabolic}) \\[2mm]
\hline
\end{tabular}
\end{table}
\FloatBarrier


\section{Notation and background}\label{section 2}
Throughout the paper, $\mathbb{N}$ denotes the set of all positive integers, $\mathbb{C}$ denotes the set of the complex numbers, $\mathbb{T}$ denotes the unit circle $\{z\in \mathbb{C}:|z|=1\}$ in the complex plane, $\mathbb{N}_0= \mathbb{N}\cup \{0\}$ and $\widehat{\mathbb{C}}=\mathbb{C}\cup\{\infty\}$. Let $\phi:\Omega\to \Omega$ be a function and $n\in \mathbb{N}$, the $n$-th iterate of $\phi$ is $\phi^{[n]}=\phi\circ \phi\circ\cdots \circ \phi$ ($n$ times). For convenience, we write $\phi^{[0]}$ for the identity function.  If $\Omega$ is an open subset of $\mathbb{C},$ we denote the set of all complex-valued holomorphic functions on $\Omega$ by $\mathrm{Hol}(\Omega).$

Let $X$ be a complex Banach space and $T:X\to X$ be a bounded linear operator. We will call $T$ an operator (instead of a bounded linear operator) and denote by $\mathcal{B}(X)$ the set of all  operators on $X.$ For the next definitions we assume $T\in \mathcal{B}(X)$ and   write $T^n$ instead of $T^{[n]}.$ The norm of $T$ is $\|T\|=\sup\{\|Tx\|:x\in X \ \text{and} \ \|x\|\leq 1\}.$ The \textit{spectrum} and the \textit{right spectrum} of $T$ are given by $\sigma(T)=\{  \lambda \in \mathbb{C}: T-\lambda I \ \text{is not invertible} \}$ and $\sigma_r(T)=\{\lambda \in \mathbb{C}: T-\lambda I \ \text{is not right invertible} \},$ respectively. The \textit{spectral radius} of $T,$ denoted by $r(T),$ is $r(T)=\sup\{|\lambda|: \lambda\in \sigma(T)\}$ and it can be determined by the spectral radius formula $r(T)=\displaystyle \lim_{n\to \infty}  \|T^n\|^{1/n}.$ We say that   $T_1\in \mathcal{B}(X_1)$ and $T_2\in \mathcal{B}(X_2)$ are similar if there exists an invertible operator $U:X_1\rightarrow X_2$ such that $T_1=U^{-1}T_2U.$

\subsection{Shadowing in linear dynamics} Our goal in this subsection is to  establish the definitions of positive shadowing property and generalized hyperbolicity. We also present some known facts about these topics.

\begin{definition} \normalfont Let $\delta>0.$ A $\delta$-\textit{pseudo-orbit} of $T\in \mathcal{B}(X)$ is a sequence $(x_n)_{n \geq 0}$ in $X$ such that
\begin{align*}
\|Tx_n-x_{n+1}\|\leq \delta, \quad \ \text{for all} \ n\in \mathbb{N}_0.
\end{align*}
\end{definition}
For $T\in \mathcal{B}(X)$, we can construct a natural $\delta$-pseudo-orbit from a non-zero vector $x\in X,$ as follows. Define $x_0 = 0$, $h = (1/\|x\|)x$  and 
\begin{align*}
x_n= \delta  \sum_{j=0}^{n-1}T^j h  \,\ \,\  \quad n\geq 1.
\end{align*}
Then, for each $n\in\mathbb{N}_0$ we have $Tx_n=x_{n+1}- \delta h$ and hence $\|Tx_n-x_{n+1}\|=\|\delta h\|=\delta.$ This construction will be used frequently in the next section.

\begin{definition} \normalfont We say that $T\in \mathcal{B}(X)$  has the \textit{positive shadowing} property  if for every $\varepsilon>0$ there exists $\delta>0$ such that every $\delta$-pseudo-orbit $(x_n)_{n \geq 0}$ of $T$ is $\varepsilon$-shadowed  by a real orbit of $T,$ that is, there exists $ x \in X$ such that 
\begin{align*}
\|T^nx-x_n\|\leq \varepsilon, \quad \ \text{for all} \ n\in \mathbb{N}_0.
\end{align*}
\end{definition}

As usual, for an invertible operator, we define the notion of \textit{shadowing} simply
by replacing the set $\mathbb{N}_0$ by the set $\mathbb{Z}$ in the definition  of positive shadowing. In \cite[Theorem 1]{Bernardes_Alfred_Shadowing2024}, Bernardes and Peris showed that these notions coincide for invertible operators on Banach spaces.

The following result is folklore; we include a proof for the sake of completeness.

\begin{proposition}\label{similar} If $T_1\in \mathcal{B}(X_1)$ and $T_2\in \mathcal{B}(X_2)$ are similar, then $T_1$ has the positive shadowing property if and only if $T_2$ does.
\end{proposition}
\begin{proof} Let $T_1=U^{-1}T_2U$ for some invertible operator $U:X_1\to X_2.$ Suppose that $T_1$ has the positive shadowing property. Given $\varepsilon>0$, there exists $\delta>0$ such that every $\delta$-pseudo-orbit  of $T_1$ is $\varepsilon/\|U\|$-shadowed by a real orbit of $T_1.$ Let $(y_n)_{n \geq 0}$ be an arbitrary $\delta/\|U^{-1}\|$-pseudo-orbit of $T_2$ and define $x_n=U^{-1}y_n$ for $n\in \mathbb{N}_0.$ Hence
\begin{align*}
\|T_1x_n-x_{n+1}\|&=\|T_1U^{-1}y_n-U^{-1}y_{n+1}\|=\|U^{-1}T_2y_n-U^{-1}y_{n+1}\|\\
&\leq \|U^{-1}\|\|T_2y_n-y_{n+1}\|\leq \delta,
\end{align*}
for all $n\in \mathbb{N}_0.$ Therefore, $(x_n)_{n \geq 0}$ is a $\delta$-pseudo-orbit  of $T_1.$ By the choice of $\delta,$ there exists $x \in X_1$ such that $\|T_1^nx-x_n\|\leq \varepsilon/\|U\|$ for all $n\in \mathbb{N}_0.$ Let $y=Ux,$ then
\begin{align*}
\|T_2^ny-y_n\|=\|T_2^nUx-y_n\|=\|UT_1^nx-Ux_n\|\leq \|U\|\|T_1^nx-x_n\|\leq \varepsilon
\end{align*}
for all $n\in \mathbb{N}_0.$ The proof of the other implication  is similar.
\end{proof}

\begin{definition} \normalfont We say that $T\in \mathcal{B}(X)$ is \textit{generalized hyperbolic} if there exists  a direct sum decomposition $X=M \oplus N $ where $M$ and $N$ are closed subspaces of $X$ with the following properties
\begin{enumerate}
    \item $T(M) \subset M$ and $r(T_{|M}) < 1;$
    \item $T|_{N}: N \to T(N)$ is bijective, $T(N)$ is closed, $N \subset T(N)$ and $r( (T{_{|N}})^{-1})_{|N}) < 1. $
\end{enumerate}
\end{definition}

Rewriting \cite[Proposition 19]{Bernardes_Messaoudi_shadowingprincipal} in terms of the last definition, we have that each generalized hyperbolic operator has the positive shadowing property. We refer to \cite{Bernardes_Messaoudi_shadowingprincipal, Bernardes_Alfred_Shadowing2024, Bernardes-Cirilo} and \cite{Favaro,Cirilo-Gollobit} for more details on shadowing in the context of linear dynamics and on generalized hyperbolic operators, respectively.

\subsection{Linear fractional composition operators and canonical forms}

The Hardy space of the open unit disk, denoted by $H^2(\mathbb{D}),$ consists of all functions $f\in \mathrm{Hol}(\mathbb{D})$ such that
\begin{align*}
\|f\|_2:= \left(\sup_{0 < r < 1} \frac{1}{2\pi} \int_{0}^{2\pi} |f(re^{i \theta} )|^2 d\theta\right)^{1/2}<\infty. 
\end{align*}
It is well known that $H^2(\mathbb{D})$ is a Hilbert space when endowed with the inner product 
\begin{align*}
\langle f, g\rangle:=\sum_{n=0}^\infty\widehat{f}(n)\overline{\widehat{g}(n)}
\end{align*}
where $(\widehat{f}(n))_{n \geq 0}$ and $(\widehat{g}(n))_{n \geq 0}$ are the Maclaurin coefficients for $f$ and $g,$ respectively. Moreover, $H^2(\mathbb{D})$ is a reproducing kernel Hilbert space of holomorphic functions. Indeed, if $w\in \mathbb{D}$ and $f\in H^2(\mathbb{D}),$ then $f(w)=\langle f, k_w\rangle$, where the \textit{reproducing kernel function} $k_w$ is given by 
\begin{align*}
k_w(z)=\frac{1}{1-\overline{w}z}, \quad z\in \mathbb{D}.
\end{align*}
A simple computation gives $\|k_w\|_2=(1-|w|^2)^{-1/2}$ for each $w\in \mathbb{D}.$ The Cauchy-Schwarz inequality gives the pointwise estimate
$\|f\|_2\|k_w\|_2\geq |f(w)|,$  for each $f\in H^2(\mathbb{D}).$ 

Due to the Littlewood Subordination Theorem, if $\phi$ is a holomorphic self-map of $\mathbb{D},$ then $C_{\phi}$ is an operator on $H^2(\mathbb{D}).$ A natural aspect of studying the properties of the composition operator $C_{\phi}$ is the analysis of the behavior of its orbits $\mathrm{Orb}(C_{\phi},f)=\{C_{\phi}^nf:n\in \mathbb{N}_0\}$. This occurs because the $n$-th iterate of $C_{\phi}$ is $C_{\phi}^n=C_{\phi^{[n]}}.$ On the other hand, since $k_0$ is a fixed point for any composition operator on $H^2(\mathbb{D}),$ it follows that $H^2(\mathbb{D})$ cannot support composition operators that are hyperbolic or positively expansive. Therefore, results such as \cite[Theorem 1 and Proposition 5]{Bernardes_Messaoudi_shadowingprincipal}, which relate hyperbolicity, (positive) expansivity and (positive) shadowing, cannot be used directly to characterize composition operators that have the shadowing property on $H^2(\mathbb{D})$.

 Recall that a \textit{linear fractional transformation} (abbreviated by LFT) is a meromorphic bijection $\phi:\widehat{\mathbb{C}}\to \widehat{\mathbb{C}}$ of the form
\begin{align}\label{lft}
\phi(z) = \frac{az + b}{cz + d}, 
\end{align}
whose coefficients $a, b, c, d \in \mathbb{C}$ satisfy $ad - bc \neq 0.$ 
We are interested here in  $\mathrm{LFT}(\mathbb{D}),$ the subfamily of LFT consisting of self-maps of $\mathbb{D}$ as in \eqref{lft}. If $\phi\in \mathrm{LFT}(\mathbb{D}),$ we say that $C_{\phi}$ is a \textit{linear fractional composition operator}. Every $\phi\in \mathrm{LFT}(\mathbb{D})$ which is not the identity function has either one or two fixed points, and according to the location of the fixed points of $\phi,$ we classify it as follows \cite{Shapiro_Bourdon_Livro_Ciclicidade, Gallardo-Gutierrez,Shapiro_livro_compopandfunctiontheory}:
\begin{itemize}
\item \textit{Parabolic} if $\phi$ has just one fixed point, which must lie in $\mathbb{T}.$ In this case, we distinguish between two different situations for $\phi$,    parabolic automorphism (\textbf{PA}) and parabolic non-automorphism (\textbf{PNA}). 
\item \textit{Hyperbolic} if $\phi$ has two fixed points, one of which must be in $\mathbb{T}.$ In this case, we distinguish between three different situations: \textit{hyperbolic automorphism} (\textbf{HA}) if the other fixed point of $\phi$ lies in $\mathbb{T},$ \textit{hyperbolic non-automorphism of type} I (\textbf{HNA I}) if the other fixed point of $\phi$ lies  in $\widehat{\mathbb{C}}\backslash \overline{\mathbb{D}}$ and  \textit{hyperbolic non-automorphism of type} II (\textbf{HNA II}) if the other fixed point of $\phi$ lies in $\mathbb{D}.$

\item \textit{Elliptic} or \textit{Loxodromic} if $\phi$ has a fixed point in $\mathbb{D}$ and another fixed point outside $\overline{\mathbb{D}}.$ We distinguish between two different situations: if $\phi$ is an automorphism,
then $\phi$ is said to be \textit{elliptic} ($\textbf{EA}$). Otherwise, $\phi$ is said to be \textit{loxodromic} ($\textbf{LOX}$). 
\end{itemize}
It is well-known that each linear fractional composition operator on  $H^2(\mathbb{D})$ is similar to a composition operator in its \textit{canonical form}  \cite{Shapiro_Bourdon_Livro_Ciclicidade}. For instance, Bourdon and Shapiro \cite[page 28]{Shapiro_Bourdon_Livro_Ciclicidade} showed that if $\phi$ is a parabolic self-map of $\mathbb{D}$, then $C_{\phi}$ is similar to some $C_{\varphi},$ where
\begin{align*}
\varphi(z)=\frac{(2-a)z+a}{-az+(2+a)}, \quad z\in \mathbb{D},
\end{align*}
with $\mathrm{Re}(a)\geq0.$ In the case that $\phi$ is a hyperbolic self-map of $\mathbb{D}$ of type I, they showed that $C_{\phi}$ is similar to some $C_{\psi},$ where 
\begin{align*}
\psi(z)=az+1-a, \quad z\in \mathbb{D},
\end{align*}
with $a\in (0,1).$

Therefore by Proposition \ref{similar}, it is sufficient to study the canonical forms of the linear fractional self-maps of $\mathbb{D}$ to completely characterize the linear fractional composition operators on $H^2(\mathbb{D})$ that have the  positive shadowing property.

\begin{remark}
   Note that the term hyperbolic is used in distinct contexts here. There is a notion of a (generalized) hyperbolic operator and the notion of a hyperbolic  self-map of $\mathbb{D}$ introduced above. We caution the reader not to confuse these definitions. 
\end{remark}

We finish this subsection by stating the main result of this paper.

 \begin{theorem}\label{Theorem} Let $C_{\phi}$ be a linear fractional composition operator on $H^2(\mathbb{D}).$ Then $C_{\phi}$ has the positive shadowing property if and only if $\phi$ is a hyperbolic automorphism or a hyperbolic non-automorphism of type I.
 \end{theorem}

\section{Shadowing property on \texorpdfstring{$H^{2}(\mathbb{D})$}{H2(D)}}\label{section 3}

In this section we present the proof of Theorem \ref{Theorem}. We do it by studying four cases separately. In the hyperbolic non-automorphism of type I and parabolic cases, we reduce the analysis to the corresponding canonical forms.

\subsection{EA, HNA II and LOX cases} The following general result suffices to show that linear fractional composition operators induced by elliptic, hyperbolic non-automorphism of type II and loxodromic self-maps of $\mathbb{D}$ do not have the positive shadowing property.

\begin{proposition}\label{loxodromic-elliptic}
Let $\phi$ be a holomorphic self-map of $\mathbb{D}.$ If $\phi$ has a fixed point in $\mathbb{D}$ then $C_{\phi}$ does not have the positive shadowing property on $H^2(\mathbb{D}).$ 
\end{proposition}

\begin{proof} Suppose that $\alpha\in \mathbb{D}$ is the fixed point of $\phi.$ We construct a $\delta$-pseudo-orbit of $C_{\phi}$ which cannot be $\varepsilon$-shadowed by any element of $H^2(\mathbb{D})$ and any $\varepsilon>0$ in the following way: given $\delta>0,$ we choose $f\in H^2(\mathbb{D})$ with $f(\alpha)\neq 0$ and we consider the natural  $\delta$-pseudo-orbit of $C_{\phi}$ constructed from  $h = (1/\|f\|_2)f;$ that is, $f_0 = 0$, $f_1=\delta h$ and 
\begin{align*}
f_n=\delta\sum_{j=0}^{n-1}C_{\phi}^{j}h, \quad \text{for}\ n\geq2.
\end{align*}
Since $h(\phi^{[j]}(\alpha)) =f(\alpha)/\|f\|_2$ for each $j\in \mathbb{N}_0,$ it follows that $f_1(\alpha)=\delta f(\alpha)/\|f\|_2$ and 
\begin{align*}
f_n(\alpha)= \left( \delta \sum_{j=0}^{n-1}C_{\phi}^{j}h \right)(\alpha)=\frac{\delta f(\alpha)}{\|f\|_2}n, \quad \text{for}\ n\geq2.
\end{align*}
Hence by the Cauchy-Schwarz inequality:
\begin{align*}
\|C_{\phi}^ng-f_n\|_2\|k_{\alpha}\|_2& \geq 
|g(\alpha)-f_n(\alpha)|= \left| \frac{\delta f(\alpha)}{\|f\|_2}n-g(\alpha) \right| \geq \frac{\delta |f(\alpha)|}{\|f\|_2}n-|g(\alpha)|
\end{align*}
which implies
\begin{align}\label{fixedpoint}
\|C_{\phi}^ng-f_n\|_2\geq (1-|\alpha|^2)^{1/2}\left[  \frac{\delta |f(\alpha)|}{\|f\|_2}n-|g(\alpha)|\right]
\end{align}
for every $g\in H^2(\mathbb{D}).$ Taking $n\to \infty$ in \eqref{fixedpoint}, we see that $\|C_{\phi}^ng-f_n\|_2\to \infty.$ This shows that $(f_n)_{n\geq 0}$ cannot be $\varepsilon$-shadowed for any $\varepsilon>0.$
\end{proof}

\begin{corollary}\label{elliptic-loxodormic-hyperbolic non-automorphism of type II} Let $\phi$ be an elliptic, a hyperbolic non-automorphism of type II or a loxodromic self-map of $\mathbb{D}.$ Then $C_{\phi}$ does not have the positive shadowing property on $H^{2}(\mathbb{D}).$ 
\end{corollary}

\subsection{PA and PNA cases} In this subsection, it is enough to consider the case that $\phi$ has the form
\begin{align}\label{parabolic symbol}
\phi(z)=\frac{(2-a)z+a}{-az+(2+a)}, \quad z\in \mathbb{D},
\end{align}
where $a$ is a non-zero complex number with $\mathrm{Re}(a)\geq0,$ since each composition operator on $H^2(\mathbb{D})$ induced by a parabolic self-map of $\mathbb{D}$ is similar to a composition operator induced by a symbol as in \eqref{parabolic symbol}. We can inductively show that the iterates of \eqref{parabolic symbol} are given by
\begin{align}\label{equality1}
\phi^{[n]}(z)=\frac{(2-na)z+na}{-naz+(2+na)}, \quad z\in \mathbb{D}. 
\end{align}

Before proceeding further, we need an auxiliary result.

\begin{lemma}\label{lemma} 
Let $a$ be a non-zero complex number with $\mathrm{Re}(a)\geq 0$ and $s\in (0,1)$. There exists a constant $c > 0$ depending on $s$ and $a$ such that 
\begin{align} \left|\sum\limits_{j = 0}^{n} (2 + ja)^s \right|\geq c n^{s + 1}
\end{align}
for every $n \in \mathbb{N}$.
\end{lemma}

\begin{proof} Given $j\in  \mathbb{N}_0$  there is $\theta_j\in (-\frac{\pi}{2}, \frac{\pi}{2})$ such that $2+ja=|2+ja|e^{i\theta_j}.$ In particular, for $s\in(0,1)$ we have $(2+ja)^s=|2+ja|^se^{is\theta_j}$ and $s\theta_j\in(-\frac{s\pi}{2}, \frac{s\pi}{2}). $ Since $\cos(\cdot)$ is a non-increasing function in $[0, \frac{\pi}{2}]$ and $s\theta_j<\frac{s\pi}{2}<\frac{\pi}{2},$ it follows that $\cos(s\theta_j)\geq \cos (\frac{s\pi}{2}):=c_1>0.$ The fact that $\cos(\cdot)$ is an even function implies that $\cos(s \theta_j) \geq c_1$ for all $\theta_j \in ( - \frac{\pi}{2}, \frac{\pi}{2})$. This allows us to obtain the following inequality
\begin{align}\label{inequality2}
\left|\sum\limits_{j = 0}^{n} (2 + ja)^s \right| & \geq \left|\mathrm{Re} ( \sum\limits_{j = 0}^{n} (2 + ja)^s )\right| = \left|\sum\limits_{j = 0}^{n} \mathrm{Re}  ((2 + ja)^s)  \right|\nonumber \\
& =  \sum\limits_{j = 0}^{n} |2 + ja|^s \cos(s \theta_j) {\geq} \sum\limits_{j = 0}^{n} |ja|^s \cos(s \theta_j)\nonumber \\ 
& \geq c_1 |a|^{s} \sum\limits_{j = 0}^{n} j^{s},
\end{align}
where in the penultimate inequality we used that $a \neq 0$ and $\mathrm{Re}(a) \geq 0$. On the other hand, since $j^s \geq \displaystyle \int_{j - 1}^{j} x^{s}dx,$ it follows that
\begin{align}\label{inequality3}
\sum_{j = 1}^{n} j^s \geq \sum_{j = 1}^{n}  \left( \int_{j - 1}^{j} x^sdx \right) = \int_{0}^{n} x^sdx = \frac{n^{s + 1}}{s + 1}.
\end{align}
Therefore by \eqref{inequality2} and \eqref{inequality3}, we obtain the estimate
\begin{align*}
\left|\sum\limits_{j = 0}^{n} (2 + ja)^s \right| \geq \frac{c_1 |a|^{s}}{s + 1} n^{s + 1},
\end{align*}
which completes the proof.
\end{proof}

 We are now ready for the main result of this subsection.

\begin{theorem}\label{parabolic} Let $\phi$ be a parabolic self-map of $\mathbb{D}.$ Then  $C_{\phi}$ does not have the positive shadowing property on $H^{2}(\mathbb{D}).$ 
\end{theorem}

\begin{proof} Suppose that $\phi$ has the form 
\begin{align*}
\phi(z)=\frac{(2-a)z+a}{-az+(2+a)}, \quad z\in \mathbb{D},
\end{align*}
where $a \neq 0$ and $\mathrm{Re}(a)\geq 0.$ We will construct a $\delta$-pseudo-orbit of $C_{\phi}$ which cannot be $\varepsilon$-shadowed by any element of $H^2(\mathbb{D})$ and any $\varepsilon>0$. Fix $s\in (0, \frac{1}{2})$, then $f(z)=(1-z)^{-s}$ belongs to $H^2(\mathbb{D})$ (see \cite[Example 1.1.14]{Rosenthal_Martinez}). Given $\delta>0,$ we consider the natural $\delta$-pseudo-orbit of $C_{\phi}$ constructed from $h=(1/\|f\|_2)f;$ that is, $f_0 = 0$, $f_1= \delta h$ and 
\begin{align*}
f_n=\delta\sum_{j=0}^{n-1}C_{\phi}^{j}h, \quad \text{for} \ n\geq 2.
\end{align*}
By using \eqref{equality1}, we can compute $f_n(0)$ precisely. Indeed, $h(0)=1/\|f\|_2$ and for $j\in \mathbb{N},$ we obtain 
\begin{align*}
h(\phi^{[j]}(0))=\frac{1}{2^s\|f\|_2}(2+ja)^s,
\end{align*}
which gives $f_1(0)=\delta /\|f\|_2$ and 
\begin{align*}
f_n(0)&=\delta\sum_{j=1}^{n-1}h(\phi^{[j]}(0))+\delta h(0)= \frac{\delta}{2^s\|f\|_2} \sum_{j=0}^{n-1}(2+ja)^s\quad \text{for} \ n\geq 2.
\end{align*}
 By the Cauchy-Schwarz inequality, we get
\begin{align*}
 \|C_{\phi}^ng-f_n\|_2\|k_0\|_2+\|g\|_2\|k_{\phi^{[n]}(0)}\|_2 \geq  |g(\phi^{[n]}(0))-f_n(0)|+ |g(\phi^{[n]}(0))| \geq |f_n(0)|.
\end{align*}
Thus,  
\begin{align*}
2\|C_{\phi}^ng-f_n\|_2\geq |f_n(0)|(1-|\phi^{[n]}(0)|^2)^{1/2}  -\|g\|_2    
\end{align*}
holds for all $n\in \mathbb{N}.$ By Lemma \ref{lemma} and the following inequality
\begin{align*}
1-|\phi^{[n]}(0)|^2&=1 - \frac{n^2 |a|^2}{(2 + n \mathrm{Re}(a))^2 + (n \mathrm{Im}(a))^2} = \frac{4 + 4 n \mathrm{Re}(a)}{ 4 + 4 n \mathrm{Re}(a) + n^2 |a|^2}\\
&\geq \frac{4}{n^2(2+|a|)^2}
 \quad (n\in \mathbb{N})
\end{align*}
we conclude that there exists a constant $c >0$ (depending on $s$ and $a$) such that
\begin{align}\label{inequality4}
\|C_{\phi}^ng-f_n\|_2\geq \frac{\delta c}{2^s\|f\|_2(2+|a|)}\frac{(n - 1)^{s + 1}}{n}-\frac{\|g\|_2}{2},
\end{align}
for every $g\in H^2(\mathbb{D})$ and all $n\in \mathbb{N}.$ Taking $n\to \infty$ in \eqref{inequality4}, we see that $\|C_{\phi}^ng-f_n\|_2\to \infty.$ This shows that $(f_n)_{n \geq 0}$ cannot be $\varepsilon$-shadowed for any $\varepsilon>0.$
\end{proof}

\subsection{HA case}  We will use the recent characterization by Pituk \cite{Pituk}.

\begin{teo}\label{Pituk} \normalfont (\cite[Theorem 1.7]{Pituk}) Let $X$ be a complex Hilbert space and $T\in \mathcal{B}(X)$ be an invertible operator. Then $T$ has the shadowing property if and only if $\sigma_r(T)\cap \mathbb{T}=\emptyset.$ 
\end{teo}

Nordgren, Rosenthal and Wintrobe \cite[Corollary 5.9]{MR899654} showed that if $\phi$ is a hyperbolic automorphism of $\mathbb{D}$ then $C_{\phi}-\lambda I$ is surjective for each $\lambda$  in the interior of $\sigma(C_{\phi}).$
Furthermore, it was shown in \cite[Theorem 7.4]{Cowen} that $\phi$ has a unique fixed point $\alpha\in \mathbb{T}$ such that $0<\phi'(\alpha)<1$ and the spectrum of $C_{\phi}$  is:
\begin{align}\label{spectrum}
\sigma(C_{\phi})=\{\lambda\in \mathbb{C}:\phi'(\alpha)^{1/2}\leq |\lambda|\leq \phi'(\alpha)^{-1/2}\}.
\end{align}

\begin{proposition}\label{hyperbolic automorphism} Let $\phi$ be a hyperbolic automorphism of $\mathbb{D}.$ Then $C_{\phi}$ has the positive shadowing property on $H^2(\mathbb{D}).$
\end{proposition}

\begin{proof} Since $C_{\phi}$ is invertible, it follows that the notions of shadowing and positive shadowing coincide \cite[Theorem 1]{Bernardes_Alfred_Shadowing2024}. Then by Theorem \ref{Pituk}, it is enough  to show that $\sigma_r(C_{\phi})\cap\mathbb{T}=\emptyset.$ If $\lambda \in \mathbb{T},$ then \eqref{spectrum} tells us that $\lambda$ is in the interior of $\sigma(C_{\phi}).$ Hence $C_{\phi}-\lambda I$ is surjective. Since each surjective operator on a Hilbert space has a bounded right inverse \cite[Chap. XI, Proposition 1.1, p. 347]{MR1070713}, it follows that $\lambda\notin \sigma_r(C_{\phi}).$ Therefore, we conclude that $C_{\phi}$ has the shadowing property.
\end{proof}

\subsection{HNA I case}  In this subsection, it is enough to consider
\begin{align}\label{hyperbolic-non}
\phi(z)=az+1-a,\quad z\in \mathbb{D},
\end{align}
where $a\in (0,1)$. As the positive shadowing property is invariant under similarity, we will transform the problem of deciding whether $C_{\phi}$ has the positive shadowing property on $H^2(\mathbb{D})$ into the problem of deciding whether a certain operator similar to $C_{\phi}$ has the positive shadowing property. We now introduce the auxiliary operators and
Hilbert spaces that play a key role in our study.

Let $\mathbb{C}_+$ denote the open right half-plane $\{z\in \mathbb{C}:\mathrm{Re}(z)>0\}.$ The Hardy space of the open right half-plane $H^2(\mathbb{C}_+)$ is the Hilbert space of all $f\in \mathrm{Hol}(\mathbb{C}_+)$ for which 
\begin{align*}
\|f\|_{H^2(\mathbb{C}_+)}:= \left( \sup_{0<x<\infty}\int_{-\infty}^{\infty}|f(x+iy)|^2dy\right)^{1/2} < \infty.  
\end{align*}
The only linear fractional self-maps of $\mathbb{C}_+$ that induce bounded composition operators on $H^2(\mathbb{C}_+)$ are those of the form 
 \begin{align}\label{fractional self-map of the right half-plane}
 \psi(w)=aw+b, \quad w\in \mathbb{C}_+,
 \end{align} where $a>0$ and $\mathrm{Re}(b)\geq0$ (see \cite{Elliot-Jury}). Henceforth, for a  fixed $a \in (0,1)$, let us denote by $\psi_a$ the self-map of $\mathbb{C}_+$ given by $\psi_a
(w)=a^{-1}w+(a^{-1}-1).$
 
Recall that $L^2(\mathbb{R}_+, dt)$ is the Hilbert space of all square-integrable measurable complex-valued functions on  $\mathbb{R}_+,$ where $\mathbb{R}_+=\{t\in \mathbb{R}:t>0\}.$ For simplicity, we write $L^2$ instead of $L^2(\mathbb{R}_+, dt)$.  By the \textit{Paley-Wiener theorem}, the transformation defined by  
\begin{align*}
(PF)(w)=\int_{\mathbb{R}_+}F(t)e^{-tw}dt
\end{align*}
is an invertible operator from $L^2$ onto $H^2(\mathbb{C}_+).$ In \cite[Lemma 9]{Carmo-Noor}, Carmo and Noor showed that if $\psi$ is as in \eqref{fractional self-map of the right half-plane}, then $aC_{\psi}:H^2(\mathbb{C}_+)\to H^2(\mathbb{C}_+)$ is similar, via the operator $P,$ to the operator $W: L^2\to L^2$ given by $(WF)(t)=e^{-bt/a}F(t/a).$ In particular, it follows that $a^{-1}C_{\psi_a}$ on $H^2(\mathbb{C}_+)$ is similar to the operator $W_a:L^2\to L^2$ where
\begin{align*}
(W_aF)(t)=e^{-t(1-a)}F(at).
\end{align*}
On the other hand, it is known that each bounded composition operator $C_{\psi}$ on $H^2(\mathbb{C}_+)$ is similar, via an invertible operator $\mathcal{U},$ to the \textit{weighted composition operator} $W_{\Phi}$ on $H^2(\mathbb{D})$ defined by
\begin{align}\label{operator}
(W_{\Phi}f)(z)=\frac{1-\Phi(z)}{1-z}f(\Phi(z)),
\end{align}
where $\Phi:=\gamma^{-1}\circ \psi\circ \gamma$ and $\gamma$ is the \textit{Cayley transform} from $\mathbb{D}$ onto $\mathbb{C}_+$ defined by $\gamma(z)=\frac{1+z}{1-z}.$ Now we note that if $a\in (0,1)$ and $\phi$ is as in \eqref{hyperbolic-non} then $\Phi(z)=(\gamma^{-1}\circ\psi_a\circ \gamma)(z)=\phi(z)$ and $\frac{1-\Phi(z)}{1-z}=a$. This allows us to conclude that $C_{\psi_a}$ is similar to $aC_{\phi}$, or equivalently, $a^{-1}C_{\psi_a}$ is similar to $C_{\phi}$ (for instance, see \cite[Lemma 14]{Carmo-Noor}). In particular, $C_{\phi}$ is also similar to $W_a.$ We have summarized these observations in the following commutative diagram:
\[
\xymatrix{
H^2(\mathbb{D}) \ar[r]^{C_{\phi}} \ar[d]_{\mathcal{U}} & H^2(\mathbb{D}) \ar[d]^{\mathcal{U}} \\
H^2(\mathbb{C_+})\ar[r]_{a^{-1}C_{\psi_a}} \ar[d]_{P^{-1}} & H^2(\mathbb{C_+}) \ar[d]^{P^{-1}}\\
L^2 \ar[r]_{W_a} & L^2
}
\]

We are now ready to state the main result of this subsection.

\begin{theorem}\label{Teo Difícil} If $a\in (0,1)$ then $W_a$ is a generalized hyperbolic operator.
\end{theorem}

\begin{proof}  Observe that $L^2$ is the direct sum of the following closed subspaces $M=\{F\in L^2:F \ \text{vanishes a.e. on}\ (0,a) \}$ and $N=\{F\in L^2:F\ \text{vanishes a.e. on}\ [a, \infty) \}.$  We will show that $M$ and $N$ satisfy the conditions of the definition of generalized hyperbolicity. 

\begin{claim}
$W_a(M)\subset M$ and $r({W_a}_{|M})<1.$ 
\end{claim}
It is easy to see that $W_a(M)\subset M.$ Inductively, we see that the $n$-th iterate of $W_a$ is given by
\begin{align*}
(W_a^nF)(t)=e^{-t(1-a^n)}F(a^nt),
\end{align*}
and by the change of variable $s=a^nt$ we obtain the following estimate
\begin{align*}
\|W_a^n F\|^2_{L^2} & = \int_{\mathbb{R}_+} |(W_a^nF)(t)|^2 dt =\int_{\mathbb{R}_+} |e^{-t(1 - a^n)} F(a^nt)|^2 dt\\
&=\frac{1}{a^n} \int_{\mathbb{R}_+} e^{-2s( \frac{1}{a^n} - 1)} |F(s)|^2 ds  =  \frac{1}{a^n} \int_{a}^{\infty} e^{-2s( \frac{1}{a^n} - 1)} |F(s)|^2 ds \\
&\leq  \frac{e^{-2a( \frac{1}{a^n} - 1)}}{a^n}  \|F\|_{L^2}^2
\end{align*}
for all $F\in M.$ Hence, for all $n \in \mathbb{N}$
\begin{align*}
\|(W_a)^n_{|M}\|^{\frac{1}{n}} \leq  \frac{e^{-\frac{a}{n}( \frac{1}{a^n} - 1)}}{\sqrt{a}} .
\end{align*}
 Since $a\in(0,1),$ it follows that $\frac{a}{n}( \frac{1}{a^n} - 1) =  \frac{ 1 - a^n}{na^{n - 1}} \to  \infty$ as $n \to \infty.$ Therefore, $r({W_a}_{|M}) = 0$.

\begin{claim} ${W_a}|_{N}: N \to W_a(N)$ is a bijection, $W_a(N)$ is closed,  $N \subset W_a(N)$ and $r((W_a{_{|N}})^{-1})_{|N})< 1.$ 
\end{claim}

Since $W_a$ is injective, it follows that ${W_a}_{|N}$ is injective and hence ${W_a}|_{N}: N \to W_a(N)$ is a bijection. In order, let us show  that $W_a(N)$ is closed. We will do this by showing that $W_a(N)=\{F\in L^2:F\ \text{vanishes a.e. on}\ [1, \infty) \}.$ Indeed,  for $F\in N$ we have $F(at)=0$ for a.e. $t\geq1$ because $at \geq a$. Thus $W_aF$ vanishes on $[1,\infty).$ This gives the inclusion $``\subset".$ On the other hand, given $F\in L^2$ that vanishes a.e. on $[1,\infty),$ we define the following function
\begin{align}\label{surjective}
G(t)=\begin{cases}
e^{\frac{1-a}{a}t}F(t/a), & 0<t<a,\\
0,& t\geq a.
\end{cases}
\end{align}
By the change of variable $s=t/a,$ we get
\begin{align*}
\int_{\mathbb{R}_+}|G(t)|^2dt&=\int_{0}^{a} e^{\frac{2(1-a)}{a}t}|F(t/a)|^2dt = a\int^1_0 e^{2s(1-a)}|F(s)|^2ds\\
&\leq ae^{2(1-a)}\int^1_0|F(s)|^2ds\leq ae^{2(1-a)}\|F\|_{L^2}^2
\end{align*}
which shows that $G\in L^2.$ Moreover, we have $G\in N$ and
\begin{align*}
(W_aG)(t)&=e^{-t(1-a)}G(at)=\begin{cases}
e^{-t(1-a)}e^{\frac{1-a}{a}at}F(t), & 0<t<1,\\
0,& t\geq 1
\end{cases}\\
&=F(t).
\end{align*}
Since $W_aG=F,$ we obtain the inclusion $``\supset ".$ In particular, it follows that $N\subset W_a(N).$ Now, for each $F\in N$ we define
\begin{align*}
(V_aF)(t)=e^{\frac{1-a}{a}t}F(t/a).
\end{align*}




Note that $V_aF$ vanishes a.e. on $[a,\infty)$ (indeed, it vanishes in $[a^2, \infty) \supset [a, \infty)$) and  by the change of variable $s=t/a$ we get  
\begin{align*}
\int_{\mathbb{R}_+}|(V_aF)(t)|^2dt&=\int_{\mathbb{R}_+}e^{\frac{2t}{a}(1-a)}|F(t/a)|^2dt=a\int_{\mathbb{R}_+}e^{2s(1 - a)}|F(s)|^2ds\\
&= a\int_{0}^{a} e^{2s(1-a)}|F(s)|^2ds\leq ae^{2a(1-a)}\|F\|^2_{L^2},
\end{align*}
which shows that $V_a: N \to N$ is a well defined operator. Moreover, a simple computation gives ${W_a}_{|N} \circ V_a=I$  and hence 
\begin{align*}
({W_a}_{|N})^{-1}F=({W_a}_{|N})^{-1}\circ({W_a}_{|N}) \circ V_aF = V_aF
\end{align*}
for each $F\in N.$ This allows us to conclude that $\left( ({W_a}_{|N})^{-1} \right)_{|N}=V_a.$ From this operator equality, we compute the iterates of $({W_a}_{|N})^{-1}_{|N},$ namely:
\begin{align*}
(V_a^nF)(t)=e^{\frac{t}{a^n}(1-a^n)}F(t/a^n).
\end{align*}
We now estimate $\|V_a^n\|.$ By using the change of variable $s=t/a^n$, we obtain
\begin{align*}
\|V_a^nF\|^2_{L^2}&=\int_{\mathbb{R}_+}|(V_a^nF)(t)|^2dt=\int_{\mathbb{R}_+}e^{\frac{2t}{a^n}(1-a^n)}|F(t/a^n)|^2dt\\
&=a^n\int_{\mathbb{R}_+}e^{2s(1-a^n)}|F(s)|^2ds=a^n\int_{0}^ae^{2s(1-a^n)}|F(s)|^2ds\\
&\leq a^ne^{2a(1-a^n)}\|F\|^2_{L^2}
\end{align*}
for all $F\in N,$ which implies $\|V_a^n\|\leq a^{\frac{n}{2}}e^{a(1-a^n)}.$ Therefore, $r(V_a) \leq \sqrt{a}<1.$ This completes the proof.
\end{proof}

\begin{corollary}\label{W_a} If $a\in (0,1)$ then $W_a$ has the positive shadowing property.
\end{corollary}

This allows us to resolve the remaining case.

\begin{corollary}\label{hyperbolic non-automorphism of type I}  Let $\phi$ be a hyperbolic non-automorphism of $\mathbb{D}$ of type I. Then $C_{\phi}$ has the positive shadowing property on $H^2(\mathbb{D}).$
\end{corollary}

\section{Comments on the shadowing property on \texorpdfstring{$H^{p}(\mathbb{D}),$ for $1\leq p\leq \infty$}{Hp(D)} }\label{section 4}

In this section, we briefly study the shadowing phenomenon for composition operators on the Hardy spaces $H^p(\mathbb{D}),$ where $1\leq p\leq\infty.$ First, let us recall that given $f\in \mathrm{Hol}(\mathbb{D}),$ $\|f\|_p$ is defined by
\begin{align*}
\|f\|_p:=\left(\sup_{0 < r < 1} \frac{1}{2\pi} \int_{0}^{2\pi} |f(re^{i \theta} )|^p d\theta\right)^{1/p}, \quad \text{for} \ 1\leq p<\infty,
\end{align*}
and $\|f\|_{\infty}:=\sup\{|f(z)|:z\in \mathbb{D}\}.$ The Hardy space $H^p(\mathbb{D})$ is the Banach space of all functions $f\in \mathrm{Hol}(\mathbb{D})$ for which the norm $\|f\|_p$ is finite.

\subsection{The case $p = \infty$}
The following result shows that $H^{\infty}(\mathbb{D})$ does not support composition operators with the positive shadowing property.

\begin{proposition} Let $\phi$ be a holomorphic self-map of $\mathbb{D}.$ Then $C_\phi$ does not have the positive shadowing property on $H^{\infty}(\mathbb{D}).$ 
\end{proposition}
\begin{proof} We construct a $\delta$-pseudo-orbit of $C_{\phi}$ which cannot be $\varepsilon$-shadowed by any element of $H^{\infty}(\mathbb{D})$ and any $\varepsilon>0$ in the following way: fix $\delta > 0$ and let $f_n$ be the constant function  $f_n (z) =  n \delta$ for $ z\in \mathbb{D}.$ Note that $(f_n)_{n\geq 0}$ is a $\delta$-pseudo-orbit of $C_{\phi}.$ By using the $H^{\infty}(\mathbb{D})$-norm, we obtain
\begin{align*}
\|f_n-C_{\phi}^nf\|_{\infty} \geq |f_n(0)-f(\phi^{[n]}(0))|\geq n\delta-|f(\phi^{[n]}(0))|\geq n\delta -\|f\|_{\infty},
\end{align*}
for every $f\in H^{\infty}(\mathbb{D}).$ This implies that $(f_n)_{n \geq 0}$ cannot be $\varepsilon$-shadowed for any $\varepsilon>0$.
\end{proof}

\subsection{The case $p \in [1,\infty)$}

When $1 \leq p < \infty$ and $p \neq 2$, the situation is a little more complicated. However, many ideas can be adapted with changes essentially related to the absence of the inner product structure. 

When $\phi$ is \textbf{LOX}, \textbf{EA}, or \textbf{HNA II}, Proposition \ref{loxodromic-elliptic} shows that $C_{\phi}$ does not have the positive shadowing property on $H^2(\mathbb{D})$. The same result holds on $H^p(\mathbb{D})$, as stated in the next theorem. In the proof for $H^2(\mathbb{D})$, we choose a function $f\in H^2(\mathbb{D})$ such that $f(\alpha)\neq 0$, where $\alpha$ is a fixed point of $\phi$, and construct a $\delta$-pseudo-orbit that cannot be $\varepsilon$-shadowed by any element of $H^2(\mathbb{D})$, for any $\varepsilon>0$. The same proof can be adapted to $H^p(\mathbb{D})$ by replacing the Cauchy--Schwarz inequality with the pointwise estimate
\begin{align}\label{pointwiseHp}
    |f(z)| \leq \frac{\|f\|_{p}}{(1 - |z|^2)^{\frac{1}{p}}}, \quad z\in \mathbb{D},
\end{align}
and starting with a function $f \in H^p(\mathbb{D})$ such that $f(\alpha) \neq 0$. This shows how to construct many $\delta$-pseudo-orbits which cannot be $\varepsilon$-shadowed for any $\varepsilon > 0$. In particular, the existence of one function $f \in H^p(\mathbb{D})$ satisfying $f(\alpha) \neq 0$ (for example some constant function) is enough. The following theorem is a consequence of this discussion and the details are left to the interested reader.

\begin{theorem}
Let $\phi$ be a holomorphic self-map of $\mathbb{D}$ and $1 \leq p < \infty.$ If $\phi$ has a fixed point in $\mathbb{D}$ then $C_{\phi}$ does not have the positive shadowing property on  $H^p(\mathbb{D})$.
\end{theorem}

The parabolic case (\textbf{PA} or \textbf{PNA}) in $H^2(\mathbb{D})$ was solved using Lemma \ref{lemma} and Theorem \ref{parabolic}. The lemma is a purely algebraic statement and does not involve any property of $H^2(\mathbb{D})$. To adapt the theorem, we need to change the test functions a bit; we consider now  $f_{s}$ where $$f_{s}(z) = (1 - z)^{-s}$$
and $s$ is a real number. It is well known that $f_{s} \in H^p(\mathbb{D})$ if and only if $s < \frac{1}{p}$. There is, however, another problem when one tries to adapt the proof, which motivates the following division of cases:

\begin{theorem}\label{Teo para H^p's} Let $\phi$ be a parabolic self-map of $\mathbb{D}$.
\begin{enumerate}[label=(\arabic*), font=\normalfont]
    \item If $\phi$ is \textbf{PA} and $1 < p < \infty$ then $C_{\phi}$ does not have the positive shadowing property on $H^p(\mathbb{D})$.
    \item If $\phi$ is \textbf{PNA} and $1 \leq p < \infty$ then $C_{\phi}$ does not have the positive shadowing property on $H^p(\mathbb{D})$.
\end{enumerate}
\end{theorem}

\begin{proof} Suppose that $\phi$ has the form 
\begin{align*}
\phi(z)=\frac{(2-a)z+a}{-az+(2+a)}, \quad z\in \mathbb{D},
\end{align*}
where $a \neq 0$ and $\mathrm{Re}(a)\geq 0.$ We will construct a $\delta$-pseudo-orbit of $C_{\phi}$ which cannot be $\varepsilon$-shadowed by any element of $H^p(\mathbb{D})$ and any $\varepsilon>0$. For each $s\in (0, \frac{1}{p})$, given $\delta>0,$ we consider the natural $\delta$-pseudo-orbit of $C_{\phi}$ constructed from $h:=(1/\|f_s\|_p)f_s;$ that is, $f_0 = 0$, $f_1= \delta h$ and 
\begin{align*}
f_n=\delta\sum_{j=0}^{n-1}C_{\phi}^{j}h, \quad \text{for} \ n\geq 2.
\end{align*}
As in the proof of Theorem \ref{parabolic} we have
\begin{align*}
|f_n(0)| &= \left| \delta\sum_{j=1}^{n-1}h(\phi^{[j]}(0))+\delta h(0) \right| = \frac{\delta}{2^s\|f_s\|_p} \left| \sum_{j=0}^{n-1}(2+ja)^s \right| \geq C (n - 1)^{s + 1}.
\end{align*}
Moreover
\begin{align*}
1-|\phi^{[n]}(0)|^2&=1 - \frac{n^2 |a|^2}{(2 + n \mathrm{Re}(a))^2 + (n \mathrm{Im}(a))^2} = \frac{4 + 4 n \mathrm{Re}(a)}{ 4 + 4 n \mathrm{Re}(a) + n^2 |a|^2}.\\
\end{align*}

Now, suppose for the sake of contradiction that $(f_n)_{n\geq 0}$ is $\varepsilon$-shadowed by some $g\in H^p(\mathbb{D})$, i.e., $\|C_{\phi}^ng-f_n\|_p\leq\varepsilon$ for all $n\in\mathbb{N}_0$. From \eqref{pointwiseHp} we obtain
\[
    |f_n(0)|(1-|\phi^{[n]}(0)|^2)^{\frac{1}{p}} \leq \varepsilon(1-|\phi^{[n]}(0)|^2)^{\frac{1}{p}}+\|g\|_p \leq \varepsilon+\|g\|_p.
\]
Hence, the sequence $$\left(|f_n(0)|(1-|\phi^{[n]}(0)|^2)^{\frac{1}{p}}\right)_{n\geq0}$$ is bounded. Now, let us analyze each case separately. Since the computation above works for every fixed $s\in(0,\frac{1}{p})$, we will choose a specific $s\in(0,\frac{1}{p})$ to obtain a contradiction.

If $\phi$ is \textbf{PA}, then $\mathrm{Re}(a) = 0$ and $ 1-|\phi^{[n]}(0)|^2 \geq c_0 n^{-2}$ for all sufficiently large $n$. Therefore
$$|f_n(0)|(1 - |\phi^{[n]}(0)|^2)^{\frac{1}{p}} \geq c_1 (n - 1)^{s + 1} n^{- \frac{2}{p}} \geq c_2 n^{s + 1 - \frac{2}{p}},$$
for some constants $c_i,$ where $i=0,1,2.$ If $p > 1$ we can choose $s \in (\max\{0, \frac{2}{p} - 1 \}, \frac{1}{p})$. This choice ensures that $s \in (0, \frac{1}{p})$ and gives a contradiction because $s + 1 - \frac{2}{p} > 0$ and thus $n^{s + 1 - \frac{2}{p}} \to \infty$.

If $\phi$ is \textbf{PNA} then $\mathrm{Re}(a) > 0$ and $ 1-|\phi^{[n]}(0)|^2 \geq c_0 n^{-1}$ (for $n$ big enough) which implies that 
$$|f_n(0)|(1 - |\phi^{[n]}(0)|^2)^{\frac{1}{p}} \geq c_1 (n - 1)^{s + 1} n^{- \frac{1}{p}} \geq c_2 n^{s + 1 - \frac{1}{p}}$$
for some constants $c_i,$ where $i=0,1,2.$ In this case, we can obtain a contradiction for $1 \leq p < \infty$ by choosing, for example, $s = \frac{1}{2p} \in (0, \frac{1}{p}).$
\end{proof}

The preceding result leaves the following question open.

\begin{ques}\normalfont
If $\phi$ is a parabolic automorphism of $\mathbb{D}$, does $C_{\phi}$ have the (positive) shadowing property on $H^1(\mathbb{D})?$
\end{ques}

For $\phi$ of type \textbf{HA} or \textbf{HNA I}, our method does not extend directly. For \textbf{HA} we used a chain of results which are valid only for Hilbert spaces. For \textbf{HNA I} we used the map $\mathcal{U}$ to move our problem to the Hardy space of the open right half-plane $H^2(\mathbb{C}_+)$ and also the Paley-Wiener theorem to move it to $L^2$. Therefore, if $p \neq 2$ and $\phi$ is \textbf{HA} or \textbf{HNA I}, the following remains open.

\begin{ques}\normalfont
If $\phi$ is a hyperbolic automorphism or a hyperbolic non-automorphism of type I, and if $1\leq p<\infty$ with $p\neq 2$, does $C_{\phi}$ have the positive shadowing property on $H^p(\mathbb{D})$?
\end{ques}

\vspace{0.3cm}

 \textit{Acknowledgments:} The authors would like to thank the anonymous referee for their careful reading and valuable comments that helped to improve the final version of this paper.


\printbibliography
\end{document}